\documentclass{article}
\usepackage{amsmath,amsfonts,amsthm,amssymb,amscd,color,xcolor}

\usepackage{hyperref}

\colorlet{darkblue}{blue!50!black}

\hypersetup{
    colorlinks,%
    citecolor=darkblue,%
    filecolor=red,%
    linkcolor=darkblue,%
    urlcolor=magenta,%
    pdfnewwindow=true,%
    pdfstartview={FitH}
}

\newcommand{\p}{\partial}
\newcommand{\e}{\varepsilon}

\newcommand{\R}{{\mathbb R}}

\newcommand{\ttau}{{\boldsymbol\tau}}
\newcommand{\barr}{{\boldsymbol|}}

\newcommand{\EE}{{\cal E}}
\newcommand{\FF}{{\cal F}}

\newcommand{\HH}{{\cal H}}
\newcommand{\II}{{\cal I}}
\newcommand{\KK}{{\cal K}}
\newcommand{\JJ}{{\cal J}}
\newcommand{\LL}{{\cal L}}

\newcommand{\XX}{{\cal X}}

\newcommand{\ZZ}{{\cal Z}}

\newcommand{\nnn}{{\boldsymbol{\mathit n}}}

\newcommand{\esssup}{\mathop{\rm ess\ sup}\limits}

\newcommand{\supp}{\mathop{\rm supp}\nolimits}

\theoremstyle{plain}
\newtheorem*{mt}{Main Theorem}
\newtheorem{theorem}{Theorem}[section]
\newtheorem{lemma}[theorem]{Lemma}
\newtheorem{proposition}[theorem]{Proposition}
\newtheorem{corollary}[theorem]{Corollary}
\theoremstyle{definition}

\newtheorem{condition}[theorem]{Condition}
\newtheorem{problem}[theorem]{Problem}
\theoremstyle{remark}

\numberwithin{equation}{section}

\begin{document}
\author{Ka\"\i s Ammari\footnote{D\'epartement de Math\'ematiques, Facult\'e des Sciences de Monastir, Universit\'e de Monastir, 5019 Monastir, Tunisie; E-mail: Kais.Ammari@fsm.rnu.tn}
\and Thomas Duyckaerts\footnote{D\'epartement de Math\'ematiques, Institut Galil\'ee, Universit\'e Paris 13, 99 avenue Jean-Baptiste Cl\'ement, 93430 Villetaneuse, France; E-mail: duyckaer@math.univ-paris13.fr}
\and Armen Shirikyan\footnote{Department of Mathematics, University of Cergy--Pontoise, CNRS UMR 8088, 2 avenue Adolphe Chauvin, 95302 Cergy--Pontoise, France; E-mail: Armen.Shirikyan@u-cergy.fr}
}
\title{Local feedback stabilisation to a non-stationary solution for a damped non-linear wave equation}
\date{}
\maketitle
\begin{abstract}
We study a damped semi-linear wave equation in a bounded domain of~$\R^3$ with smooth boundary. It is proved that any $H^2$-smooth solution can be stabilised locally by a finite-dimensional feedback control supported by a given open subset satisfying a geometric condition. The proof is based on an investigation of the linearised equation, for which we construct a stabilising control satisfying the required properties. We next prove that the same control  stabilises locally the non-linear problem. 


\end{abstract}

\tableofcontents

\setcounter{section}{-1}

\section{Introduction}
\label{s0}
Let us consider the damped non-linear wave equation (NLW)
\begin{equation} \label{1}
\p_t^2u+\gamma\p_t u-\Delta u+f(u)=h(t,x), \quad x\in \Omega,
\end{equation}
where $\Omega\subset\R^3$ is a bounded domain with a smooth boundary~$\p \Omega$, $\gamma>0$ is a parameter, $h$ is a locally square-integrable function with range in~$L^2(\Omega)$, and $f:\R\to\R$ is a function satisfying some natural growth and regularity conditions ensuring the existence, uniqueness, and regularity of a solution.  Equation~\eqref{1} is supplemented with the Dirichlet boundary condition,
\begin{equation} \label{2}
u\bigr|_{\p \Omega}=0,
\end{equation}
and the initial conditions
\begin{equation} \label{3}
u(0,x)=u_0(x), \quad \p_tu(0,x)=u_1(x),
\end{equation}
where $u_0\in H_0^1(\Omega)$ and $u_1\in L^2(\Omega)$. It is well known that, even though problem~\eqref{1}, \eqref{2} is dissipative and possesses a global attractor (which is  finite-dimensional in the autonomous case), its flow is not locally stable, unless we impose very restrictive conditions on the nonlinear term~$f$. That is, the difference between two solutions with close initial data, in general, grows in time. The purpose of this paper is to show that any sufficiently regular solution of~\eqref{1}, \eqref{2} can be stabilised with the help of a finite-dimensional feedback control localised in space. Namely, instead of~\eqref{1}, consider the controlled equation
\begin{equation} \label{4}
\p_t^2u+\gamma\p_t u-\Delta u+f(u)=h(t,x)+\eta(t,x), \quad x\in \Omega,
\end{equation}
where $\eta$ is a control whose support in~$x$ contains a subset satisfying a  geometric condition. For a time-dependent function~$v$, we set 
$$
\varPhi_v(t)=\bigl[v(t),\dot v(t)\bigr], \quad
E_v(t)=\int_\Omega\bigl(|\dot v(t,x)|^2+|\nabla_x v(t,x)|^2\bigr)\,dx,
$$
where the dot over a function stands for its time derivative. The following theorem is an informal statement of the main result of this paper.

\begin{mt}
Let~$\hat u$ be the solution of~\eqref{1}--\eqref{3} with sufficiently regular initial data and let~$\omega\subset\Omega$ be a neighbourhood of the boundary~$\p\Omega$. Then there is a finite-dimensional subspace $\FF\subset H_0^1(\omega)$ and a family of continuous operators $K_{\hat u}(t):H_0^1\times L^2\to \FF$ such that, for any $[u_0,u_1]\in H_0^1\times L^2$ that is sufficiently close to $\varPhi_{\hat u}(0)$, problem \eqref{2}--\eqref{4} with $\eta(t)=K_{\hat u}(t)\varPhi_{u-\hat u}(t)$ has a unique solution~$u(t,x)$, which satisfies the inequality
\begin{equation} \label{5}
E_{u-\hat u}(t) \le Ce^{-\beta t}E_{u-\hat u}(0) , \quad t\ge0,
\end{equation}
where $C$ and~$\beta$ are positive numbers not depending on~$[u_0,u_1]$.
\end{mt}

We refer the reader to Section~\ref{s4} (see Theorem~\ref{t4.1}) for the exact formulation of our result. Before outlining the main idea of the proof of this theorem, we discuss some earlier results concerning the semi-linear wave equation with localised control. This problem  was studied in a number of works, and first local results were obtained by Fattorini~\cite{fattorini-1975} for rectangular domains and Chewning~\cite{chewning-1976} for a bounded interval. Zuazua~\cite{zuazua-1990,zuazua-1993} proved that the 1D wave equation with a nonlinear term~$f(u)$ growing at infinity no faster that $u(\log u)^2$ possesses the property of global exact controllability, provided that the control time is sufficiently large. These results were extended later to the multidimensional case, as well as to the case of nonlinearities with a faster growth and ``right'' sign at infinity~\cite{zhang-2000,LZ-2000}. Furthermore, the question of stabilisation to the zero solution was studied in~\cite{zuazua-cpde1990}. 

The application of methods of microlocal analysis enables one to get sharper results. In the linear case, the exact controllability and stabilisation to a stationary solution are established under conditions that are close to being necessary; e.g., see~\cite{BLR-1992,LR-1997}. Dehman, Lebeau, and Zuazua~\cite{DLZ-2003} proved the global exact controllability by a control supported in the neighbourhood of the boundary, provided that nonlinearity is subcritical and satisfies the inequalities $f(u)u\ge0$; see also~\cite{DL-2009} for a refinement of this result. Fu, Yong, and Zhang~\cite{FYZ-2007} established a similar result for the equations with variable coefficients and a nonlinearity~$f$ growing at infinity slower than~Ê$u(\log u)^{1/2}$. Coron and Tr\'elat~\cite{CT-2006} proved exact controllability of 1D nonlinear wave equation in a connected component of steady states. Laurent~\cite{laurent-2011} extended the result of~\cite{DLZ-2003} to the case of the nonlinear Klein--Gordon equation with a critical exponent. Very recently, Joly and Laurent~\cite{JL-2012} proved the global exact controllability of the NLW equation of a particular form, using a fine analysis of the dynamics on the attractor. In conclusion, let us mention that the problem of exact controllability and stabilisation for other type of semi-linear dispersive equations was investigated in a large number of works, and we refer the reader to the  papers~\cite{DGL-2006,laurent-2010} for an overview of the literature  in this direction. 

\medskip
To the best of our knowledge, the problem of stabilisation of a nonstationary solution~$\hat u$ by a finite-dimensional localised control was not studied earlier. Without going into detail, let us describe informally the main idea of our approach, which is based on the study of Eq.~\eqref{1} linearised around~$\hat u$. We thus consider the equation 
\begin{equation} \label{6}
\p_t^2u+\gamma\p_t u-\Delta u+b(t,x)u=\eta(t,x), \quad x\in \Omega,
\end{equation}
supplemented with the initial and boundary conditions~\eqref{2} and~\eqref{3}. We wish to find a finite-dimensional localised control depending on the initial conditions $[u_0,u_1]$  such that the energy~$E_u(t)$ goes to zero exponentially fast. Following a well-known idea coming from the theory of attractors (see~\cite{haraux-1985} and Section~II.6 in~\cite{BV1992}), we represent a solution of~\eqref{6}, \eqref{3} in the form $u=v+w$, where~$v$ is the solution of~\eqref{6} with $b\equiv\eta=0$  issued from~$[u_0,u_1]$. Then~$w$ satisfies the zero initial conditions and Eq.~\eqref{6} with~$\eta$ replaced by $\eta-bv$. Let us note that~$v$ goes to zero in the energy space exponentially fast and that~$w$ has better regularity properties than~$v$. It follows, in particular, that the decay of a sufficiently large finite-dimensional projection of~$w$ will result in exponential stabilisation of~$u$. Combining this with the general scheme used in~\cite{BRS-2011} and a new observability inequality established in Section~\ref{s3}, we construct a finite-dimensional localised control~$\eta$ which squeezes to zero the energy norm of~$w$ and that of~$u$. A standard technique enables one to prove the latter property can be achieved by a feedback control. We refer the reader to Section~\ref{s2} for more details. 

\smallskip
The paper is organised as follows. In Section~\ref{s1}, we recall some well-known results on the Cauchy problem for semi-linear wave equation and establish the existence and uniqueness of a solution for the linear problem with low-regularity data. Section~\ref{s2} is devoted to the stabilisation to zero for the linearised equation by a finite-dimensional localised control. The key tool for proving this result is the truncated observability inequality established in Section~\ref{s3}. Finally, the main result on local stabilisation for the non-linear problem is presented in Section~\ref{s4}. 

\medskip
{\bf Acknowledgements}. This work was initiated when the third author was visiting The University of Monastir (Tunisia) in April of 2011. He thanks the institute for hospitality. 
The research of TD was supported by the ERC starting grant {\it DISPEQ\/}, the ERC advanced grant {\it BLOWDISOL\/} (No.~291214), and the ANR JCJC grant {\it  SchEq\/}.  
The research of AS supported by the Royal Society--CNRS grant {\it Long time behavior of solutions for stochastic Navier--Stokes equations\/} (No.~YFDRN93583) and the ANR grant {\it STOSYMAP\/} (No.~ANR 2011 BS01 015 01). 

\subsection*{Notation}
Let $J\subset\R$ be a closed interval, let~$\Omega\subset\R^d$ be a bounded domain with $C^\infty$ boundary~$\p \Omega$, and let~$X$ and~$Y$ be  Banach spaces. We shall use the following functional spaces.

\medskip
\noindent
$C_0^\infty(\Omega)$ is the space of infinitely smooth functions $f:\Omega\to\R$ with compact support. 

\smallskip
\noindent
$L^p=L^p(\Omega)$ is the usual Lebesgue space of measurable functions $f:\Omega\to\R$ such that
$$
\|f\|_{L^p}^p:=\int_\Omega|f(x)|^pdx<\infty.
$$
When $p=2$, this norm is generated by the $L^2$-scalar product~$(\cdot,\cdot)$ and is denoted by~$\|\cdot\|$. 

\smallskip
\noindent
$H^s=H^s(\Omega)$ is the Sobolev space of order~$s$ with the standard norm~$\|\cdot\|_s$. 

\smallskip
\noindent
$H_0^s=H_0^s(\Omega)$ is the closure of $C_0^\infty(\Omega)$ in~$H^s$. It is well known that $H_0^s=H^s$ for $s<\frac12$ and that~$H^s$ with $s\le0$ is the dual of~$H_0^{-s}$ with respect to~$(\cdot,\cdot)$. 

\smallskip
\noindent
$C(J,X)$ is the space of continuous functions $f:J\to X$ with the topology of uniform convergence on bounded intervals. Similarly, $C^k(J,X)$ is the space of~$k$ time continuously differentiable functions $f\in C(J,X)$. 

\smallskip
\noindent
$L^p(J,X)$ is the space of Bochner measurable functions $f:J\to X$ with a finite norm
$$
\|f\|_{L^p(J,X)}=\biggl(\int_J\|f(t)\|_X^pdt\biggr)^{1/p}.
$$
In the case $p=\infty$, this norm should be replaced by $\esssup_{t\in J}\|f(t)\|_X$. 

\smallskip
\noindent
$W^{k,p}(J,X)$ is the space of functions $f\in L^p(J,X)$ such that $\p_t^j f\in L^p(J,X)$ for $j=0,\dots,k$. 

\smallskip
\noindent
$\LL(X,Y)$ denotes the space of continuous linear operators from~$X$ to~$Y$ with the usual operator norm~$\|\cdot\|_{\LL(X,Y)}$. When the choice of~$X$ and~$Y$ is clear, we simply write~$\|\cdot\|_\LL$. 

\medskip
\noindent
Given a function of the time variable $v(t)$, we write $\varPhi_v(t)=[v(t),\dot v(t)]$, where the dot stands for the time derivative. 

\smallskip
\noindent
We denote by $C$ unessential numbers (which may vary from line to line) and by~$M_i(a_1,\dots,a_n)$  positive constants depending on the parameters $a_1,\dots,a_n$. 

\smallskip
\noindent
We write $J_T=[0,T]$, $J_T(s)=[s,s+T]$, $\R_+=[0,+\infty)$, and $R_s=[s,+\infty)$. 

\section{Preliminaries}
\label{s1}
In this section, we first recall a well-known result on the existence, uniqueness, and regularity of a solution for~\eqref{1}--\eqref{3}. We next turn to the linearised problem, for which we prove the well-posedness in a space of functions of low regularity. Finally, we derive some commutator estimates used in what follows. 

\subsection{Cauchy problem for a semi-linear wave equation}
\label{s1.1}
Let us consider the equation
\begin{equation}
\p_t^2u+\gamma\p_t u-\Delta u+f(u)=g(t,x), \quad x\in \Omega,\label{1.1}
\end{equation}
where $\Omega\subset\R^3$ is a bounded domain with the boundary $\p\Omega\in C^\infty$, $\gamma>0$ is a parameter, $g$ is a locally square-integrable function of time with range in~$L^2(\Omega)$,  and~$f\in C^1(\R)$ is a function vanishing at $u=0$ and satisfying the inequalities
\begin{equation} \label{1.2}
|f'(u)|\le C(1+|u|^2), \quad F(u):=\int_0^uf(v)\,dv\ge -C, \quad u\in\R. 
\end{equation}
Equation~\eqref{1.1} is supplemented with the initial and boundary conditions~\eqref{2} and~\eqref{3}. A proof of the following  result on the well-posedness of the initial-boundary value problem for~\eqref{1.1} and regularity of solutions can be found in Chapter~1 of~\cite{lions1969} and Section~1.8 of~\cite{BV1992}. 

\begin{proposition} \label{p1.1}
Under the above-mentioned hypotheses, for any $u_0\in H_0^1(\Omega)$, $u_1\in L^2(\Omega)$, and $g\in L_{\rm loc}^2(\R_+,L^2)$, problem~\eqref{1.1}, \eqref{2}, \eqref{3} has a unique solution 
\begin{equation} \label{1.3}
u\in C(\R_+,H_0^1)\cap C^1(\R_+,L^2)\cap W_{\rm loc}^{2,2}(\R_+,H^{-1}). 
\end{equation}
Moreover, if in addition $u_0\in H^2(\Omega)$, $u_1\in H_0^1(\Omega)$, and $\p_tg\in L_{\rm loc}^2(\R_+,L^2)$, then 
\begin{equation} \label{1.4}
u\in C(\R_+,H^2)\cap C^1(\R_+,H_0^1)\cap W_{\rm loc}^{2,2}(\R_+,L^2). 
\end{equation}
\end{proposition}

We now formulate a result on the time boundedness of solution for~\eqref{1.1} under some additional assumptions on~$f$. Namely, let us assume that $f\in C^2(\R)$ is such that 
\begin{equation} \label{1.5}
f(u)u\ge c\,F(u)-C, \quad f'(u)\ge -C,\quad |f''(u)|\le C(1+|u|),\quad u\in\R,
\end{equation}
where $C$ and $c$ are positive constants.  Note that these conditions are satisfied for polynomials of degree~$3$ with positive leading coefficient and, more generally, for $C^2$-smooth functions behaving at infinity as $c|u|^{\rho-1}u$ with $c>0$ and $\rho\in[1,3]$. The following result is established by Zelik~\cite{zelik-2004}.

\begin{proposition} \label{p1.2}
Let us assume that~$f\in C^2(\R)$ satisfies~\eqref{1.2} and~\eqref{1.5} and let $g\in W^{1,\infty}(\R_+,L^2)$. 
Then, for any $u_0\in H^2(\Omega)\cap H_0^1(\Omega)$ and $u_1\in H_0^1(\Omega)$, the solution $u(t,x)$ of~\eqref{1.1}, \eqref{2}, \eqref{3} satisfies the inequality
\begin{equation} \label{1.7}
\|u(t)\|_2+\|\dot u(t)\|_1\le M_1\bigl(\|u_0\|_2,\|u_1\|_1,G\bigr), \quad t\ge0,
\end{equation}
where we set
\begin{equation*} \label{1.6}
G:=\esssup_{t\ge 0}\bigl(\|g(t,\cdot)\|+\|\p_tg(t,\cdot)\|\bigr)<\infty.
\end{equation*}
\end{proposition}

\subsection{Cauchy problem for the wave equation with low-regu\-larity data}
\label{s1.2}
In this and the next subsections, we assume that the space dimension~$d\ge1$ is arbitrary, even though the results obtained here will be used only for $d=3$. We study the linearised problem
\begin{align} 
\ddot v+\gamma\dot v-\Delta v+b(t,x)v&=\eta(t,x),\quad x\in \Omega,\label{2.1}\\
v\bigr|_{\p \Omega}&=0,\label{2.2}\\
v(0,x)=v_0(x), \quad \dot v(0,x)&=v_1(x), 
\label{1.12}
\end{align}
where $b$ and~$\eta$ are functions of low regularity, and  $\gamma\in\R$. Namely, let~$\{e_j\}$ be the complete set of $L^2$-normalised eigenvectors for the Dirichlet Laplacian in~$\Omega$ (denoted by~$-\Delta$) and let~$\{\lambda_j\}$ be the corresponding eigenvalues numbered in an increasing order. We define the scale of spaces associated with~$-\Delta$ by the relation
\begin{align}
H_D^s=H_D^s(\Omega)=\Bigl\{f\in L^2(\Omega):\sum_{j\ge1}(f,e_j)^2\lambda_j^s<\infty\Bigr\}\quad\mbox{for $s\ge0$}
\end{align}
and denote by~$H_D^{-s}=H_D^{-s}(\Omega)$ the dual of~$H_D^s$ with respect to the $L^2$ scalar product. It is well known that (see~\cite{fujiwara-1967})
\begin{equation} \label{1.14}
H_D^s=H_0^s\quad\mbox{for $\frac12<s<\frac12$}, \qquad
H_D^s=H^s\quad\mbox{for $-\frac32<s<\frac12$}. 
\end{equation}

The wave propagator for~\eqref{2.1}, \eqref{2.2} is well defined (with the help of the eigenfunction expansion) when $b\equiv\eta\equiv0$. In this case, for any initial data $[v_0,v_1]\in \HH_D^s:=H_D^s\times H_D^{s-1}$ there is a unique solution 
$$
v\in C(\R,H_D^s)\cap C^1(\R,H_D^{s-1}),
$$
which satisfies the inequality
\begin{equation} \label{1.15}
\|\varPhi_v(t)\|_{\HH_D^s}\le e^{c|t|}\|\varPhi_v(0)\|_{\HH_D^s}
\quad\mbox{for $t\in\R$},
\end{equation}
where $c\ge0$ depends only on~$\gamma$. We denote by $S(t):\HH_D^s\to\HH_D^s$ the operator taking~$[v_0,v_1]$ to~$\varPhi_v(t)$ and write $S(t)=[S_0(t),S_1(t)]$, so that~$S_0(t)$ is a continuous operator from~$\HH_D^s$ to~$H_D^s$ for any $s\in\R$. 

Let us define $\HH^s=H_0^s\times H_0^{s-1}$ and note that, in view of~\eqref{1.14}, we have $\HH^s=\HH^s_D$ for $s\in(-\frac12,\frac12)$. The proof of the following result is rather standard and is based on the Duhamel representation, the Banach fixed point theorem, and an estimate for the Sobolev of the product of two functions.

\begin{proposition} \label{p1.3}
Let $r$ and $T$ be positive numbers, let $J_T=[0,T]$, and let $b\in L^\infty(J_T\times\Omega)\cap L^\infty(J_T,H^r)$. Then there is $\sigma_0(r)>0$ such that for any $\sigma\in[0,\sigma_0(r)]$, $[v_0,v_1]\in\HH^{-\sigma}$, and $\eta\in L^1(J_T,H^{-\sigma-1})$ problem~\eqref{2.1}--\eqref{1.12} has a unique solution
$$
v\in X_T^\sigma:=C(J_T,H^{-\sigma})\cap C^1(J_T,H^{-\sigma-1}). 
$$
Moreover, there is $M_2=M_2\bigl(T,\|b\|_{L^\infty(J_T,L^\infty\cap H^r)}\bigr)$ such that
\begin{equation} \label{1.16}
\sup_{t\in J_T}\|\varPhi_v(t)\|_{\HH^{-\sigma}}
\le M_2\Bigl(\bigl\|[v_0,v_1]\bigr\|_{\HH^{-\sigma}}
+\bigl\|\eta\bigr\|_{L^1(J_T,H^{-\sigma-1})}\Bigr).
\end{equation}
\end{proposition}

Note that inequality~\eqref{1.16} is true for $\sigma=-1$, provided that $b\in L^\infty(J_T\times\Omega)$ for any $T>0$; this is a simple consequence of the standard energy estimate for the wave equation and the Gronwall inequality. 

\begin{proof}[Proof of Proposition~\ref{p1.3}]
We need to construct a solution of the integral equation
\begin{equation}
 \label{Duhamel_b}
v(t)=S_0(t)[v_0,v_1]-\int_0^t S_0(t-s)\bigl[0,b(s)v(s)\bigr]\,ds
+\int_0^t S_0(t-s)\bigl[0,\eta(s)\bigr]\,ds, 
\end{equation}
where $t\in J_T$. We shall prove the existence of a solution on small time interval $J_\tau=[0,\tau]$ whose length does not depend on the size of the initial data and the right-hand side. The global existence will then follow by iteration. 

\smallskip
{\it Step~1. Bound on the Duhamel term}. 
Let us set 
\begin{equation*}
Qg(t)=\int_0^t S_0(t-s)[0,g(s)]\,ds.
\end{equation*} 
We first show that if $g\in L^1(J_\tau,H^{-1-\sigma})$, then 
\begin{equation} \label{bdt}
Qg\in X_{\tau}^\sigma, \quad \|Qg\|_{X_{\tau}^\sigma}\leq C_1\|g\|_{L^1(J_\tau,H^{-1-\sigma})},
\end{equation}
where~$C_1$ does not depend on~$g$. 
Indeed, if $0\leq \sigma<\frac 12$, then $H^{-\sigma-1}=H_D^{-\sigma-1}$ and thus the mapping $s\mapsto \{S_0(t-s)[0,g(s)],t\in J_\tau\}$ belongs to the space $L^1(J_\tau,Y_\tau^\sigma)$, where $Y_\tau^\sigma=C(J_\tau,H^{-\sigma})$. 
Furthermore, by~\eqref{1.15}, we have 
$$ 
\|S_0(t-s)[0,g(s)]\|_{-\sigma}\le e^{c\tau}\|g(s)\|_{-\sigma-1}
\quad\mbox{for all $t,s\in J_\tau$},
$$
It follows that $Qg\in C(J_\tau,H^{-\sigma})$ and 
$$
\|Qg(t)\|_{-\sigma}\leq C\,\|g\|_{L^1(J_\tau,H^{-\sigma-1})}\quad
\mbox{for $t\in J_\tau$}.
$$
Using the relation
$$
\frac{\partial}{\partial t} S_0(t-s)[0,g(s)]=S_0(t-s)[g(s),0],
$$
we see that $\p_t(Qg)\in C(J_\tau,H^{-\sigma-1})$ and 
$$
\|\partial_t Qg(t)\|_{-\sigma-1}\leq C\,\|g\|_{L^1(J_\tau,H^{-\sigma-1})}, 
\quad t\in J_\tau.
$$
This completes the proof of~\eqref{bdt}.

\smallskip
{\it Step~2. Fixed point argument}. 
We shall need the following lemma, whose proof is given at the end of this subsection. 

\begin{lemma} \label{l1.4}
Let $r\in[0,1+d/2]$, let $\Omega\subset\R^d$ be a bounded domain with smooth boundary, and let $a\in L^\infty\cap H^r$. Then, for any $s\in[0,r]$, we have
\begin{align} 
\|af\|_s&\le C\,\|a\|_{L^\infty\cap H^r}\|f\|_{ps}\quad \mbox{for $f\in H^{ps}$},
\label{1.17}\\
\|af\|_{-sp}&\le C\,\|a\|_{L^\infty\cap H^r}\|f\|_{-s}\quad \mbox{for $f\in H^{-s}$},
\label{1.18}
\end{align}
where $p=\frac{d+2}{2r}$, and $C>0$ depends only on~$\Omega$. 
\end{lemma}

For $v\in X_{\tau}^\sigma$, we define
\begin{equation} \label{def_Pv}
Pv(t)=\int_0^{t} S_0(t-s)\bigl[0,b(s)v(s)\bigr]\,ds,\quad t\in J_\tau.
\end{equation} 
Using~\eqref{1.18}, one can find a number $\sigma_0(r)>0$ such that, if $0<\sigma\le\sigma_0(r)$, then $bv \in L^{\infty}(J_\tau,H^{-\sigma-1})$ and
$$ 
\|bv\|_{L^{\infty}(J_\tau,H^{-\sigma-1})}\leq 
C\,\|b\|_{L^{\infty}(J_\tau,L^{\infty}\cap H^r)}\|v\|_{X_{\tau}^\sigma}.
$$
Hence, in view of~\eqref{bdt}, we have 
\begin{equation} \label{bound_Pv}
\|Pv\|_{X_{\tau}^\sigma}\leq 
C\,\tau\left\|b\right\|_{L^{\infty}(J_\tau,H^r\cap L^{\infty})}\|v\|_{X_{\tau}^\sigma}.
\end{equation} 
Now note that~\eqref{Duhamel_b} holds for all $t\in J_\tau$ and a function $v\in X_\tau^\sigma$ if and only if
$$ 
(\mathrm{Id}+P)u=S(t)[v_0,v_1]+Q\eta,\quad u\in X_{\tau}^\sigma.
$$
Choosing~$\tau$ so small that 
$$
C\tau\left\|b\right\|_{L^{\infty}(0,\tau;H^r\cap L^{\infty})}\leq \frac{1}{2},
$$
we see from~\eqref{bound_Pv} that $\mathrm{Id}+P$ is invertible in $X_{\tau}^\sigma$ and 
\begin{equation} \label{1.22}
\left\|(\mathrm{Id}+P)^{-1}\right\|_{\LL(X_{\tau}^\sigma,X_{\tau}^\sigma)}\le2.
\end{equation}
We thus obtain the existence of a unique solution $v\in X_\tau^\sigma$ for~\eqref{Duhamel_b}, which can be represented in the form
$$
v=(\mathrm{Id}+P)^{-1}\left(S(\cdot)[v_0,v_1]+Q\eta\right).
$$
Combining this with~\eqref{1.15}, \eqref{bdt}, and~\eqref{1.22}, we obtain the required estimate~\eqref{1.16}. This completes the proof of the proposition.
\end{proof}

\begin{proof}[Proof of Lemma~\ref{l1.4}]
Let us consider the multiplication operator $f\mapsto af$. 
Using the continuity of the embedding $H^{1+d/2}\subset L^\infty$, the fact that~$H^s$ is a Banach algebra for $s>d/2$, and interpolation techniques, it is easy  to prove that 
$$
\|af\|_r\le C_1\|a\|_r\|f\|_{1+d/2}, \quad f\in H^{1+d/2}.
$$
On the other hand, it is obvious that
$$
\|af\|\le \|a\|_{L^\infty}\|f\|, \quad f\in L^2.
$$
By interpolation, the above two inequalities imply that 
$$
\|af\|_{\theta r}\le C_1^\theta \|a\|_r^\theta \|a\|_{L^\infty}^{1-\theta}\|f\|_{\theta(1+d/2)}, \quad f\in H^{\theta(1+d/2)}.
$$
Taking $\theta=s/r$, we arrive at  inequality~\eqref{1.17}. 

To prove~\eqref{1.18}, note that $p\ge1$, whence it follows that the operator of multiplication by~$a$ sends $H_0^{sp}$ to~$H_0^s$. By duality, it is also continuous from~$H^{-s}$ to~$H^{-ps}$, and inequality~\eqref{1.18} is implied by~\eqref{1.17}. 
\end{proof}

\subsection{Commutator estimates}
\label{s1.3}
Given a function $\psi\in C_0^\infty(\R)$, we define 
$$
\psi(-\Delta)f=\sum_{j=1}^\infty \psi(\lambda_j)(f,e_j)e_j, \quad f\in L^2.
$$
The aim of this subsection is to derive some estimates for the commutator of~$\psi(-\Delta)$ with the multiplication operator. In what follows, given a function~$a\in L^2(\Omega)$, we denote by the same symbol the corresponding multiplication operator sending $f$ to~$af$. We begin with the case of a smooth function. 

\begin{lemma}
\label{L:com1}
 Let $a\in C^{\infty}(\overline{\Omega})$, $\psi\in C_0^{\infty}(\R)$, and $\alpha\in [0,1/2)$. Then there is $M_3=M_3(a,\psi)$ such that, for $h\in(0,1]$, we have
\begin{align}
\label{commut1}
\left\|\left[\psi(-h^2\Delta),a\right]\right\|_{\LL(H^{\alpha+1}_0,H^{\alpha})}
&\leq M_3 h^2,\\
\label{commut2}
\left\|\left[\psi(-h^2\Delta),a\right]\right\|_{\LL(H^{-\alpha},L^2)}
&\leq M_3 h^{1-\alpha}.
\end{align} 
\end{lemma}

\begin{proof}
 We first prove~\eqref{commut1}. Using the Fourier inversion formula, we get
$$ 
\psi(-h^2\Delta)=\frac{1}{2\pi}\int_{\R} e^{-ish^2\Delta}\hat{\psi}(s)\,ds,
$$
where $\hat{\psi}$ is the Fourier transform of $\psi$. Thus
$$ 
\left[\psi(-h^2\Delta),a\right]f=\frac{1}{2\pi}\int_{\R} \hat{\psi}(s)v(h^2s)\,ds,
$$
where $v(s)=\left[e^{-is\Delta},a\right]f$ is the solution of the problem
\begin{equation} \label{eq_v}
-i\partial_s v+\Delta v=[a,\Delta](e^{-is\Delta}f),\quad v\bigr|_{s=0}=0.
\end{equation} 
Using the fact that $H_0^{\alpha+1}=H_D^{\alpha+1}$, we derive
\begin{align*}
 \left\|\left[a,\Delta\right]e^{-is\Delta}f\right\|_{\alpha}
 &= \left\|(\Delta \,a) e^{-is\Delta}f+2\nabla a\cdot\nabla (e^{-is\Delta}f)\right\|_{\alpha}\\
&\leq C\left\|e^{-is\Delta}f\right\|_{\alpha+1}
=C\left\|f\right\|_{\alpha+1}.
\end{align*}
Combining this with~\eqref{eq_v}, we obtain
\begin{equation*}
\|v(s)\|_{\alpha}\leq C|s|\|f\|_{\alpha+1}\quad\mbox{for all $s\geq 0$},
\end{equation*} 
whence it follows that
$$   
\left\|\left[\psi(-h^2\Delta),a\right]f\right\|_{\alpha}
\leq C\int_{\R} \|v(h^2s)\|_{\alpha}|\hat{\psi}(s)|\,ds\leq
C\,h^2\|f\|_{\alpha+1}\bigl\|s \hat{\psi}\bigr\|_{L^1}.
$$
This completes the proof of \eqref{commut1}.

\smallskip
To prove~\eqref{commut2}, we first note that inequality~\eqref{commut1} with $\alpha=0$ implies by duality that
\begin{equation}
 \label{commut1bis}
\left\|\left[a,\psi(-h^2\Delta)\right]\right\|_{\LL(L^2,H^{-1})}\leq C h^2.
\end{equation} 
Let $\varphi\in C_0^{\infty}(\R)$ be such that $\varphi=1$ on the support of $\psi$. Using the relation $\psi(-h^2\Delta)\varphi(-h^2\Delta)=\psi(-h^2\Delta)$, we get 
$$
\left[a,\psi(-h^2\Delta)\right]=\left[a,\psi(-h^2\Delta)\right]\varphi(-h^2\Delta)+\psi(-h^2\Delta)\big[a,\varphi(-h^2\Delta)\big].
$$
Combining with~\eqref{commut1}, \eqref{commut1bis}, and the inequality
$$ 
\|\varphi(-h^2\Delta)\|_{\LL(L^2,H_0^{\alpha+1})}+\|\psi(-h^2\Delta)\|_{\LL(H^{-1},H^{\alpha})}\leq Ch^{-1-\alpha},
$$
we derive
$$
\left\|\left[a,\psi(-h^2\Delta)\right]\right\|_{\LL(L^2,H^{\alpha})}
\leq C h^{1-\alpha}.
$$
By duality, we obtain~\eqref{commut2}.
\end{proof}

We now turn to the case of functions of low regularity, which will be important in the derivation of an observability inequality (see Section~\ref{s3}). 

\begin{lemma}
\label{L:com2}
For any $r>0$ there is $\sigma>0$ such that, if $\psi \in C_0^{\infty}(\R)$ and $a\in H^r(\Omega)\cap L^{\infty}(\Omega)$, then
\begin{equation} \label{1.28}
\left\|\left[\psi(-h^2\Delta),a\right]\right\|_{\LL(H^{-\sigma},H^{-1})}
\le M_4 h^{\sigma},
\end{equation}
where $M_4=M_4(\psi,\|a\|_{L^\infty\cap H^r})$. 
\end{lemma}

\begin{proof}
There is no loss of generality in assuming that $0<r\le1$. 
Let us set $A_h=\left[a,\psi(-h^2\Delta)\right]$. Since $\psi(-h^2\Delta):L^2\to L^2$ is bounded uniformly in $h\in(0,1]$, we have 
\begin{equation}
\label{estimA1}
\sup_{h\in (0,1]} \|A_h\|_{\LL(L^2,L^2)}<\infty.
\end{equation} 
Furthermore, since $\psi(-h^2\Delta):H_D^s\to H_D^s$ is also bounded uniformly in $h\in(0,1]$ for any $s\in\R$, it follows from~\eqref{1.17} that
\begin{equation}
 \label{estimA2}
\sup_{h\in (0,1]} \|A_h\|_{\LL(H^{ps}_D,H^{s}_D)}<\infty
\quad\mbox{for $s\in[0,r]$},
\end{equation} 
where $p=\frac{d+2}{2r}$. We next show
\begin{equation}
\label{estimA3}
\sup_{h\in(0,1]} h^{-2}\|A_h\|_{H_D^2\to H_D^{-2}}<\infty.
\end{equation} 
To this end, we write (cf.\ proof of Lemma~\ref{L:com1})
\begin{equation}
\label{A_h}
A_h f=\frac{1}{2\pi}\int_\R \hat{\psi}(\tau)v(h^2\tau)\,d\tau,\quad 
v(t)=ae^{-it\Delta}f-e^{-it\Delta}(af).
\end{equation} 
Now note that
$$ 
\partial_t v=i\Delta\left(e^{it\Delta}(af)\right)-ia\Delta\left(e^{it\Delta}f\right),
\quad v(0)=0.
$$
It follows that
\begin{equation*}
 \|\partial_t v\|_{H_D^{-2}}\leq \left\|e^{it\Delta}(af)\right\|_{L^2}+C\left\|\Delta e^{it\Delta}f\right\|_{L^2}\leq C\|f\|_{H_D^{2}}^2,
\end{equation*}
whence we conclude that
$$ 
\|v\|_{H_D^{-2}}\leq C|t|\|f\|_{H_D^{2}}. 
$$
Recalling~\eqref{A_h}, we get the inequality
$$ 
\|A_hf\|_{H_D^{-2}}\leq Ch^2\bigl\|\tau\hat{\psi}\bigr\|_{L^1}\|f\|_{H_D^2},
$$
which implies~\eqref{estimA3}.

\smallskip
Interpolating \eqref{estimA1} and \eqref{estimA3}, we derive
\begin{equation*}
\sup_{h\in (0,1]}
 h^{-s}\left\|A_h\right\|_{\LL(H_D^{s},H_D^{-s})}<\infty.
\end{equation*} 
Interpolating with \eqref{estimA2}, we deduce
\begin{equation*}
\sup_{h\in (0,1]}
 h^{-s/3}\left\|A_h\right\|_{H_D^{(2p+1)s/3}\to H_D^{s/3}}<\infty.
 \end{equation*} 
Taking $s=3/(2p+1)$, by duality we obtain inequality~\eqref{1.28} with $\sigma=s/3$.
\end{proof}

\section{Stabilisation of the linearised equation}
\label{s2}

This section is devoted to the stabilisation of the linearised problem~\eqref{2.1}, \eqref{2.2}, in which $\gamma>0$, $b$ is a given function, and~$\eta$ is a finite-dimensional control supported by a given subdomain of~$\Omega$. The main result of this section is the existence of a feedback control exponentially stabilising problem~\eqref{2.1}, \eqref{2.2}. To this end, we first construct a finite-dimensional stabilising control and then use a standard technique to get a feedback law.

\subsection{Main result and scheme of its proof}
As before, we denote by~$\Omega\subset\R^3$ a bounded domain with $C^\infty$ boundary~$\Gamma$. We shall always assume that the following two conditions are satisfied. 

\begin{condition} \label{c2.1}
The smooth surface~$\Gamma$ has only finite-order contacts with its tangent straight lines. 
\end{condition}
In other words, let $y\in\Gamma$ and $\ttau_y\subset\R^3$ be the tangent plane to~$\Gamma$ at the point~$y$. In a small neighbourhood of~$y$, the surface~$\Gamma$ can be represented as the graph of a smooth function~$\varphi_y:\ttau_y\to\R$ vanishing at~$y$ together with its first-order derivatives. Condition~\ref{c2.1} requires that the restriction of~$\varphi_y$ to the straight lines passing through~$y$ has no zero of infinite order at the point~$y$. 

To formulate the second condition, we first introduce some notation. Given $x_0\in\R^3$, define $\Gamma(x_0)$ as the set of points $y\in\Gamma$ such that $\langle y-x_0,\nnn_y\rangle>0$, where~$\nnn_y$ stands for the outward unit  normal to~$\Gamma$ at the point~$y$. Let~$\omega$ be the support of the control function~$\eta$ entering~\eqref{2.1}.

\begin{condition} \label{c2.2}
There is $x_0\in \R^3\setminus\overline\Omega$ and $\delta>0$ such that
\begin{equation*} \label{2.4}
\Omega_\delta(x_0):=\{x\in\Omega:\mbox{there is $y\in\Gamma(x_0)$ such that $|x-y|<\delta$}\}\subset\omega.
\end{equation*}
\end{condition}

Before formulating the main result of this section, let us make some comments about the above hypotheses. Condition~\ref{c2.2} naturally arising in the context of the multiplier method (see~\cite{lions1988}) ensures that the observability inequality holds for~\eqref{2.1} in the energy norm. On the other hand, Condition~\ref{c2.1} enables one to define a generalised bicharacteristic flow on~$\Omega$ (see Section~24.3 in~\cite{hormander1994}). Together with Condition~\ref{c2.2}, this implies that if~$T$ is sufficiently large, then for any $\delta'\in(0,\delta)$ the pair $(\Omega_{\delta'}(x_0),T)$ geometrically controls~$\Omega$ in the sense that every generalised bicharacteristic ray of length~$T$ meets the set~$\Omega_{\delta'}(x_0)$. In view of~\cite{BLR-1992}, it follows that the observability inequality holds for Eq.~\eqref{2.1} with $b\equiv0$ in spaces of negative regularity. We shall combine these two results with some commutators estimates and a compactness argument to establish a truncated observability inequality for~\eqref{2.1} (see Section~\ref{s3.2}), which is a key point of the proof of the theorem below.  

\smallskip
Let us fix a function $\chi\in C_0^\infty(\R^3)$ such that $\supp\chi\cap\Omega\subset\omega$ and $\chi(x)=1$ for $x\in\Omega_{\delta/2}(x_0)$. We denote by~$\FF_m$ the vector span of the functions  $\{\chi e_1,\dots,\chi e_m\}$, where $\{e_j\}$ is a complete set of $L^2$ normalised eigenfunctions for the Dirichlet Laplacian. The following theorem is the main result of this section.

\begin{theorem} \label{t2.1}
Let Condition~\ref{c2.1} and~\ref{c2.2} be satisfied, let~$R$ and~$r$ be positive numbers, and let $b(t,x)$ be a function such that
\begin{equation} \label{2.5}
\barr b\barr:= \esssup_{t\ge 0}\|b(t,\cdot)\|_{L^\infty\cap H^r}\le R.
\end{equation}
Then there is an integer $m\ge1$, positive numbers~$C$ and~$\beta$, and a family of continuous linear operators
$$
K_b(t):H_0^1\times L^2\to\FF_m, \quad t\ge0,
$$
such that the following properties hold. 
\begin{description}
\item[\bf Time continuity and boundedness.]
The function $t\mapsto K_b(t)$ is continuous from~$\R_+$ to the space $\LL(H_0^1\times L^2,\FF_m)$ endowed with the weak operator topology, and its norm is bounded by~$C$.

\item[Exponential decay.]
For any $s\ge0$ and $[v_0,v_1]\in H_0^1\times L^2$, problem~\eqref{2.1}, \eqref{2.2} with the right-hand side $\eta(t)=K_b(t)[v(t),\dot v(t)]$ has a unique solution $v\in C(\R_+,H_0^1)\cap C^1(\R_+,L^2)$ satisfying the  initial conditions
\begin{equation} \label{2.6}
v(s,x)=v_0(x), \quad \p_tv(s,x)=v_1(x). 
\end{equation}
Moreover, we have the inequality
\begin{equation} \label{2.7}
E_v(t)\le C\,e^{-\beta(t-s)} E_v(s), \quad t\ge s. 
\end{equation}
\end{description}
\end{theorem}

Let us sketch the proof of this result. We first prove that, for any $\beta>0$ and sufficiently large~$T>0$, there is a linear operator $\Theta_s:\HH\to L^2(J_T(s),\FF_m)$, where $J_T(s)=[s,s+T]$ and $\HH=H_0^1\times L^2$, such that the norm of~$\Theta_s$ is bounded uniformly in~$s\ge0$, and for any $[v_0,v_1]\in H_0^1\times L^2$ the solution of problem~\eqref{2.1}, \eqref{2.2}, \eqref{2.6} with $\eta=\Theta_s[v_0,v_1]$ satisfies the inequality
\begin{equation} \label{2.8}
E_v(s+T)\le e^{-\beta T} E_v(s). 
\end{equation}
For given initial data  $[v_0,v_1]\in H_0^1\times L^2$, an exponentially stabilising control~$\eta$ can be constructed by the rule
\begin{equation} \label{2.9}
\eta\bigr|_{J_T(0)}=\Theta_0[v_0,v_1], \qquad \eta\bigr|_{J_T(kT)}=\Theta_{kT}\varPhi_v(kT), \quad k\ge1.
\end{equation}
Inequality~\eqref{2.8} and the uniform boundedness of~$\Theta_s$ imply that~\eqref{2.7} holds with $s=0$. Once the existence of at least one exponentially stabilising control is proved, one can use a standard technique based on the dynamical programming principle to construct a feedback law possessing the required properties. The uniqueness of a solution is proved by a standard argument based on the Gronwall inequality. 

\smallskip
The rest of this section is organised as follows. In Subsection~\ref{s2.2}, we prove the existence of an operator~$\Theta_s$ with the above-mentioned properties. A key point of the proof is the truncated observability inequality established in Section~\ref{s3}. Subsection~\ref{s2.3}  deals with the construction of an exponentially stabilising feedback law. Its properties mentioned in the theorem are established in Subsection~\ref{s2.4}. 
In what follows, the domain~$\Omega$ and its  closed subset~$\omega$ are assumed to be fixed, and we do not follow the dependence of other quantities on them.

\subsection{Construction of a stabilising control}
\label{s2.2}
\begin{proposition} \label{p2.4}
Let the hypotheses of Theorem~\ref{t2.1} hold and let $T>0$ be sufficiently large. Then, for a sufficiently small $\sigma>0$, there is a constant~$C$ and an integer~$m\ge1$, depending only~$R$ and $r$, such that, for any $s\ge0$, one can construct a continuous linear operator $\Theta_s: \HH\to L^2(J_T(s),H^\sigma)$ satisfying the following properties.
\begin{description}
\item[\bf Boundedness.] The norm of~$\Theta_s$ is bounded by~$C$ for any~$s\ge0$, and its image is contained in $L^2(J_T(s),\FF_m)$.
\item[\bf Squeezing.]
Let $[v_0,v_1]\in \HH$ and $\eta=\Theta_s[v_0,v_1]$. Then the solution of~\eqref{2.1}, \eqref{2.2}, \eqref{2.6} satisfies inequality~\eqref{2.8}. 
\end{description}
\end{proposition}

An immediate consequence of this proposition is the following result on the existence of a stabilising control. For $\beta>0$ and a Banach space~$X$, we denote by $L_\beta^2(\R_+,X)$ the space of locally square-integrable functions $f:\R_+\to X$ such that
$$
\|f\|_{L_\beta^2}:=\sup_{t\ge0}\int_t^{t+1}e^{\beta s}\|f(s)\|_X^2ds<\infty. 
$$

\begin{corollary} \label{c2.5}
Under the hypotheses of Theorem~\ref{t2.1}, there is~$\beta>0$ and a continuous linear operator $\Theta: \HH\to L_\beta^2(\R_+,\FF_m)$, where~$\FF_m$ is endowed with the norm of~$H^\sigma$ with some $\sigma>0$,  such that the solution of problem~\eqref{2.1}, \eqref{2.2}, \eqref{2.6} with $\eta=\Theta[v_0,v_1]$ and $s=0$ satisfies inequality~\eqref{2.7} with $s=0$. 
\end{corollary}

\begin{proof}
Let us define a control $\eta:\R_+\to\FF_m$ by relations~\eqref{2.9}. It follows from~\eqref{2.8} that
\begin{equation} \label{2.11}
E_v(kT)\le e^{-\beta T k} E_v(0), \quad k\ge0. 
\end{equation}
Since the norms of~$\Theta_s$ are bounded uniformly in~$s\ge0$, it follows from~\eqref{2.9} and~\eqref{2.11} that
\begin{equation} \label{2.12}
\bigl\|\eta|_{J_T(kt)}\bigr\|_{L^2(J_T(kT),\FF_m)}
\le C\,\|\varPhi_v(kT)\|_\HH
\le Ce^{-\beta T k/2}\bigl\|[v_0,v_1]\bigr\|_\HH. 
\end{equation}
This inequality shows that $\eta\in L_\beta^2(\R_+,\FF_m)$ and the operator $[v_0,v_1]\mapsto \eta$ is continuous from~$\HH$ to $L_\beta^2(\R_+,\FF_m)$. 
Furthermore, the continuity of the resolving operator for problem~\eqref{2.1}, \eqref{2.2} and inequalities~\eqref{2.11} and~\eqref{2.12} imply that
\begin{align*}
\sup_{t\in J_T(kT)}\|\varPhi_v(t)\|_\HH
&\le C\Bigl(\bigl\|\varPhi_v(kT)\bigr\|_\HH
+\bigl\|\eta|_{J_T(kt)}\bigr\|_{L^2(J_T(kT),\FF_m)}\Bigr)\\
&\le Ce^{-\beta T k/2}\bigl\|[v_0,v_1]\bigr\|_\HH.
\end{align*}
This immediately implies the required estimate~\eqref{2.7} with $s=0$. 
\end{proof}

\begin{proof}[Proof of Proposition~\ref{p2.4}]
We first describe the scheme of the proof. 
Define an energy-type functional for a trajectory $v(t,x)$ by the relation
$$
\EE_v(t)=\int_\Omega\bigl(|\dot v|^2+|\nabla v|^2+\alpha v\dot v\bigr)\,dx. 
$$
For small $\alpha>0$, this quantity is equivalent to~$E_v(t)$:
\begin{equation} \label{2.13}
C^{-1}E_v(t)\le \EE_v(t)\le C E_v(t), \quad t\in\R_+.
\end{equation}
Let $z$ be the solution of problem~\eqref{2.1}, \eqref{2.2}, \eqref{2.6} with $b\equiv\eta\equiv0$. Taking the scalar product in~$L^2$ of the equation for~$z$ with $2\dot z+\alpha z$, we can find $\delta>0$ such that
\begin{equation} \label{2.14}
E_z(t)\le C\EE_z(t)\le Ce^{-\delta (t-s)}\EE_z(s)\le C^2e^{-\delta (t-s)}E_z(s). 
\end{equation}
In particular, if $T>0$ is sufficiently large, then
\begin{equation} \label{2.15}
\|\varPhi_z(s+T)\|_\HH\le \frac14 \bigl\|[v_0,v_1]\bigr\|_\HH. 
\end{equation}
We seek a solution in the form $v=z+w$. Then $w$ must be a solution of the control problem
\begin{align} 
\ddot w+\gamma\dot w-\Delta w+b(t,x)w&=\eta(t,x)-b(t,x)z,\quad x\in \Omega,\label{2.16}\\
w\bigr|_{\p \Omega}&=0,\label{2.17}\\
w(s,x)=0, \quad \dot w(s,x)&=0. 
\label{2.18}
\end{align}
Given an integer $N\ge1$ and a constant~$\e>0$, we shall construct a control~$\eta$ such that the corresponding solution~$w$ satisfies the inequalities
\begin{equation} \label{2.19}
\|\varPhi_w(s+T)\|_{\HH^\sigma}\le M_5\bigl\|[v_0,v_1]\bigr\|_\HH\,, \quad 
\|{\mathsf P}_N\varPhi_w(s+T)\|\le \e\,\bigl\|[v_0,v_1]\bigr\|_\HH\,,
\end{equation}
where ${\mathsf P}_N$ stands for the orthogonal projection in~$L^2(\Omega)$ to the vector span of the first~$N$ eigenfunctions of the Dirichlet Laplacian, and~$\sigma>0$ and~$M_5$ are constants not depending on~$N$ and~$\e$. For an appropriate choice of~$N$ and~$\e$, these two inequalities imply that 
\begin{equation} \label{2.20}
\|\varPhi_w(s+T)\|_\HH\le\frac14\bigl\|[v_0,v_1]\bigr\|_\HH.
\end{equation}
Combining this with~\eqref{2.15}, we see that
$$
\|\varPhi_v(s+T)\|_\HH\le\frac12\bigl\|[v_0,v_1]\bigr\|_\HH.
$$
This inequality is equivalent to~\eqref{2.8} with $\beta=T^{-1}\log 2$. 

\medskip
We now turn to the accurate proof. The derivation of inequality~\eqref{2.15} is classical (e.g., see Section~6 in~\cite[Chapter~2]{BV1992}), and we shall confine ourselves to the construction of~$w$. To simplify notation, we shall assume that $s=0$; the case $s>0$ can be treated by a literal repetition of the argument used for $s=0$.

\smallskip
{\it Step~1}. 
We seek~$\eta$ in the form 
\begin{equation} \label{2.21}
\eta(t,x)=\chi(x){\mathsf P}_m\bigl(\zeta(t, \cdot)\bigr),
\end{equation}
where $\zeta\in L^2(J_T\times\Omega)$ is an unknown function and $m\ge1$ is an integer that will be chosen later. Let us define the space 
$$
\XX_T=C(J_T,H_0^1)\cap C^1(J_T,L^2)\cap W^{2,2}(J_T,H^{-1})
$$
and consider the following minimisation problem: 

\begin{problem} \label{p2.5}
Given initial data $[v_0,v_1]\in\HH$ and (small) positive numbers~$\delta$ and~$\sigma$, minimise the functional
$$
\JJ(w,\zeta)=\frac12\int_0^{T}\|\zeta(t,\cdot)\|_\sigma^2\,dt
+\frac1\delta\bigl(\|\nabla {\mathsf P}_Nw(T)\|^2+\|{\mathsf P}_N\dot w(T)\|^2\bigr)
$$
in the class of functions $(w,\zeta)\in\XX_T\times L^2(J_T,H^{\sigma})$ satisfying Eqs.~\eqref{2.16}, \eqref{2.18} with $s=0$ and~$\eta$ given by~\eqref{2.21}.
\end{problem}
This is a linear-quadratic optimisation problem, and it is straightforward to prove the existence and uniqueness of an optimal solution, which will be denoted by~$(w,\zeta)$. The mapping $z\mapsto (w,\zeta)$ is linear, and therefore so is the mapping $[v_0,v_1]\mapsto \eta$. Let us derive some estimates for the norms of~$w$ and~$\zeta$. To this end, we write the optimality conditions:
\begin{align}
\ddot q-\gamma\dot q-\Delta q+b(t,x)q&=0,\label{2.22}\\
(-\Delta)^\sigma\zeta&={\mathsf P}_m(\chi q), \label{2.23}\\
q(T)=-\frac{2}{\delta}\,{\mathsf P}_N\dot w(T), \quad 
\dot q(T)&=-\frac{2}{\delta}\,{\mathsf P}_N\bigl(\Delta w(T)+\gamma\dot w(T)\bigr),
\label{2.24}
\end{align}
where $q\in L^2(J_T,H_0^1)$ is a Lagrange multiplier. Note that, in view of~\eqref{2.22} and the uniqueness of a solution for the linear wave equation, the function~$q$ must belong to~$\XX_T$, so that relations~\eqref{2.24} make sense. Let us take the scalar product in~$L^2(J_T\times\Omega)$ of Eqs.~\eqref{2.16} and~\eqref{2.22} with the functions~$q$ and~$w$, respectively, and take the difference of the resulting equations. After some simple transformations, for $\sigma\in(0,\frac12)$ we obtain
$$
(\dot w(T)+\gamma w(T),q(T))-(w(T),\dot q(T))+\int_0^T(bz,q)\,dt-\int_0^T\bigl(\zeta,{\mathsf P}_m(\chi q)\bigr)\,dt=0. 
$$
Using~\eqref{2.23} and~\eqref{2.24}, we obtain
\begin{equation} \label{2.25}
\frac2\delta\bigl(\|{\mathsf P}_Nw(T)\|_1^2+\|{\mathsf P}_N\dot w(T)\|^2\bigr)+\int_0^T\bigl\|{\mathsf P}_m(\chi q)\bigr\|_{-\sigma}^2\,dt
=\int_0^T(bz,q)\,dt. 
\end{equation}
There is no loss of generality in assuming that~$T$ is so large that inequality~\eqref{2.07} holds and, hence, the  truncated observability inequality~\eqref{3.10} is true for small~$\sigma>0$. Combining this with~\eqref{1.17}, \eqref{1.16}, and~\eqref{2.14} we derive
\begin{align*}
|(bz,q)|&\le \|bz\|_\sigma\|q\|_{-\sigma}
\le C\|z\|_1\|\varPhi_q(0)\|_{\HH^{-\sigma}}\\
&\le C\|\varPhi_z(0)\|_{\HH}\,\bigl\|{\mathsf P}_m(\chi q)\bigr\|_{L^2(J_T,H^{-\sigma})}.
\end{align*}
The Cauchy--Schwarz inequality now implies that 
$$
\biggl|\int_0^T(bz,q)\,dt\biggr|
\le \frac12\int_0^T\bigl\|{\mathsf P}_m(\chi q)\bigr\|_{-\sigma}^2\,dt
+C\bigl(\|v_0\|_1^2+\|v_1\|^2\bigr).
$$
Substituting this into~\eqref{2.25}, we obtain
\begin{equation} \label{2.26}
\frac1\delta\bigl(\|{\mathsf P}_Nw(T)\|_1^2+\|{\mathsf P}_N\dot w(T)\|^2\bigr)+\int_0^T\bigl\|{\mathsf P}_m(\chi q)\bigr\|_{-\sigma}^2\,dt
\le C\,\bigl\|[v_0,v_1]\bigr\|_\HH^2,
\end{equation}
where $m=m(N)\ge1$ is an integer and $C$ is a constant  not depending on~$\delta$ and~$N$. Taking $N\gg1$ and $\delta\ll1$, we obtain the second inequality in~\eqref{2.19} with $s=0$.

\smallskip
{\it Step~2}. 
Let un prove the boundedness of the operator $\Theta_0:[v_0,v_1]\mapsto\eta$ from~$\HH$ to~$L^2(J_T,H^\sigma)$ and inequality~\eqref{2.20} with $s=0$. This will complete the proof of Proposition~\ref{p2.4}. 

It follows from~\eqref{2.23} and~\eqref{2.26} that
$$
\|\zeta\|_{L^2(J_T,H^\sigma)}\le C\,\bigl\|[v_0,v_1]\bigr\|_\HH. 
$$
Since the projection~${\mathsf P}_m$ and multiplication by~$\chi$ are bounded operators in~$H^\sigma$, the above inequality combined with relation~\eqref{2.21} shows that~$\Theta_0$ is bounded. 

To prove~\eqref{2.20}, we write
\begin{align*}
\|\varPhi_w(T)\|_{\HH}
&\le \|{\mathsf P}_N\varPhi_w(T)\|_{\HH}
+\|(I-{\mathsf P}_N)\varPhi_w(T)\|_{\HH}\\
&\le \e\,\bigl\|[v_0,v_1]\bigr\|_\HH
+\delta_N(\sigma)\|\varPhi_w(T)\|_{\HH^\sigma},
\end{align*}
where $\delta_N(\sigma)\to0$ as $N\to\infty$. It follows that~\eqref{2.20} will be established if we prove the first inequality in~\eqref{2.19}. 

Duhamel representation for solutions of~\eqref{2.16}--\eqref{2.18} and inequality~\eqref{1.15} with $s=\sigma$ imply that 
$$
\|\varPhi_w(T)\|_{\HH^\sigma}\le C\,\bigl\|\eta-bz-bw\bigr\|_{L^1(J_T,H^\sigma)}. 
$$
Combining this with~\eqref{1.17} and using condition~\eqref{2.5} and the boundedness of~$z$ and~$w$ in $C(J_T,H_0^1)$, we arrive at the required inequality. 
\end{proof}

\subsection{Dynamic programming principle and feedback law}
\label{s2.3}
Once the existence of a stabilising control is established, an exponentially stabilising feedback law can be constructed using a standard approach based on the dynamic programming principle. Since the corresponding argument was carried out in detail for the more complicated case of the Navier--Stokes system (see Section~3 in~\cite{BRS-2011}), we shall omit some of the proofs. Let us consider the following optimisation problem depending on the parameter~$s\ge0$. 

\begin{problem} \label{p2.6}
Given $[v_0,v_1]\in\HH$ and $\beta>0$, minimise the functional
$$
\II_s(v,\zeta)=\frac12\int_s^\infty e^{\beta t}
\bigl(\|\nabla v(t)\|^2+\|\dot v(t)\|^2+\|\zeta(t)\|^2\bigr)\,dt
$$
in the class of functions $(v,\zeta)$ such that
$$
v\in C(\R_s,H_0^1)\cap C^1(\R_s,L^2)\cap W_{\rm loc}^{2,2}(\R_s,H^{-1}), \quad
\zeta\in L_{\rm loc}^2(\R_s,L^2),
$$
and Eqs.~\eqref{2.1} and~\eqref{2.6} hold with~$\eta$ given by~\eqref{2.21}. 
\end{problem}
This is a linear-quadratic optimisation problem, and in view of Corollary~\ref{c2.5}, there is at least one admissible pair~$(v,\zeta)$ for which $\II_s(v,\zeta)<\infty$. It follows that there is a unique optimal solution $(v^s,\zeta^s)$ for Problem~\ref{p2.6}, and the corresponding optimal cost can be written as
$$
\II_s(v^s,\zeta^s)=\frac12\bigl(Q_s[v_0,v_1],[v_0,v_1]\bigr)_\HH\,,
$$
where $Q_s:\HH\to\HH$ is a bounded positive  operator in the Hilbert space~$\HH$. Moreover, repeating the argument used in the proof of Lemma~3.8 in~\cite{BRS-2011}, one can prove that~$Q_s$ continuously depends on~$s$ in the weak operator topology, and its norm satisfies the inequality
\begin{equation} \label{2.27}
\|Q_s\|_{\LL(\HH)}\le C\,e^{\beta s}, \quad s\ge0.
\end{equation}
We now consider the following problem depending on the parameter $s>0$. 

\begin{problem} \label{p2.7}
Given $[v_0,v_1]\in\HH$ and $\beta>0$, minimise the functional
$$
\KK_s(v,\zeta)=\frac12\int_0^s e^{\beta t}
\bigl(\|\nabla v(t)\|^2+\|\dot v(t)\|^2+\|\zeta(t)\|^2\bigr)\,dt
+\frac12\bigl(Q_s[v_0,v_1],[v_0,v_1]\bigr)_\HH
$$
in the class of functions $(v,\zeta)$ such that
$$
v\in C(J_s,H_0^1)\cap C^1(J_s,L^2)\cap W^{2,2}(J_s,H^{-1}), \quad
\zeta\in L^2(J_s,L^2),
$$
and Eqs.~\eqref{2.1} and~\eqref{1.12} hold with~$\eta$ given by~\eqref{2.21}. 
\end{problem}
This is a linear-quadratic optimisation problem, which has a unique solution~$(\tilde v^s,\tilde \zeta^s)$. The following lemma establishes a link between Problems~\ref{p2.6} and~\ref{p2.7}. Its proof repeats the argument used in~\cite{BRS-2011} (see Lemma~3.10) and is omitted. 

\begin{lemma} \label{p2.8}
Let $(v,\zeta)=(v^0,\zeta^0)$ be the unique solution of Problem~\ref{p2.6} with $s=0$. Then the restriction of~$(v,\zeta)$ to the interval~$J_s$ coincides with~$(\tilde v^s,\tilde \zeta^s)$ and the restriction of~$(v,\zeta)$ to the half-line~$\R_s$ coincides with~$(v^s,\zeta^s)$ corresponding to the initial data~$[v(s),\dot v(s)]$. 
\end{lemma}

The optimality conditions for the restriction of~$(v,\zeta)$ to~$J_s$ imply, in particular, that
\begin{align}
\ddot q_s-\gamma\dot q_s-\Delta q_s+b(t,x)q_s
&=e^{\beta t}\bigl(\Delta v+\beta\dot v+\ddot v\bigr),\label{2.28}\\
\zeta(t)&=e^{-\beta t}{\mathsf P}_m(\chi q_s(t)), \label{2.29}\\
q_s(s)&=-Q_s^1\varPhi_v(s),
\label{2.30}
\end{align}
where $q_s\in L^2(J_s,H_0^1)$ is a Lagrange multiplier and $Q_s^1:\HH\to L^2$ is a continuous operator defined by the relation $Q_sV=[Q_s^0V,Q_s^1V]$ for $V\in\HH$. Since the right-hand side of~\eqref{2.28} belongs to $L^2(J_s,H^{-1})$, it follows from Proposition~\ref{p1.3} with $\sigma=0$ that the function~$q_s$ must belong to $C(J_s,L^2)$, so that relation~\eqref{2.30} makes sense, and~$\zeta$ is a continuous function of time with range in~$\FF_m$. Combining~\eqref{2.29} and~\eqref{2.30}, we see that
$$
\zeta(s)=-e^{-\beta s}{\mathsf P}_m\bigl(\chi Q_s^1\varPhi_v(s)\bigr). 
$$
Recalling that $s>0$ is arbitrary and using~\eqref{2.21}, we conclude that the unique optimal solution~$(v,\zeta)$ of Problem~\ref{p2.6} with $s=0$ satisfies Eq.~\eqref{2.1} with $\eta(t,x)=K_b(t)\varPhi_v(t)$, where the linear operator $K_b(t):\HH\to\FF_m$ is given by
\begin{equation} \label{2.31}
K_b(t)V=-e^{-\beta t}\chi\,{\mathsf P}_m\bigl(\chi Q_t^1V\bigr), \quad t\ge0.
\end{equation}
In the next section, we shall show that this operator satisfies all the properties mentioned in Theorem~\ref{t2.1}. 

\subsection{Conclusion of the proof of Theorem~\ref{t2.1}}
\label{s2.4}
The continuity of the function $t\mapsto K_b(t)$ in the weak operator topology follows from a similar property for~$Q_t$, and the uniform boundedness of its norm is an immediate consequence of~\eqref{2.27}. To establish~\eqref{2.7}, we first consider the case~$s=0$. Let us define $w(t,x)=e^{\beta t/2}v(t,x)$. Then there is $C>1$ such that
\begin{equation} \label{2.32}
C^{-1}E_w(t)\le e^{\beta t}E_v(t)\le C\,E_w(t), \quad t\ge0. 
\end{equation}
Furthermore, the function~$w$ must satisfy the equation
\begin{equation} \label{2.33}
\ddot w+\gamma \dot w-\Delta w=g(t,x),
\end{equation}
where we set
$$
g(t,x)=e^{\beta t/2}\Bigl(K_b(t)\varPhi_v(t)
+\bigl(\tfrac{\beta^2}{4}+\tfrac{\gamma\beta}{2}-b\bigr)v+\beta\dot v\Bigr). 
$$
Note that $\|g(t)\|^2\le C\,e^{\beta t}E_v(t)$, and since~$v$ is the optimal solution of Problem~\ref{p2.6} with $s=0$, we have
\begin{equation} \label{2.34}
\int_0^\infty\|g(t)\|^2\,dt\le C\,\bigl(Q_0[v_0,v_1],[v_0,v_1]\bigr)_\HH
\le C\,E_v(0). 
\end{equation}
Taking the scalar product in~$L^2$ of Eq.~\eqref{2.33} with $2\dot w+\alpha w$, carrying out some standard transformations, and using the Gronwall inequality, we derive
$$
E_w(t)\le Ce^{-\delta t}E_w(0)
+C\int_0^te^{-\delta(t-\theta)}\|g(\theta,\cdot)\|^2d\theta.
$$
Combining this with~\eqref{2.32} and~\eqref{2.34}, we arrive at the required inequality~\eqref{2.7} with $s=0$. 

\smallskip
To prove~\eqref{2.7} with an arbitrary $s=\theta>0$, we repeat the above argument with the initial point moved to~$\theta$. Namely, considering an analogue of Problem~\ref{p2.7} on the half-line~$\R_\theta$, one can prove by the same argument as above that 
$$
\zeta^\theta(t)
=-e^{-\beta t}{\mathsf P}_m\bigl(\chi Q_t^1\varPhi_v^\theta(t)\bigr),
\quad t>\theta. 
$$
It follows that if~$v$ is the solution of~\eqref{2.1}, \eqref{2.6} with  $\eta(t)=K_b(t)\varPhi_v(t)$ and $s=\theta$, then 
\begin{align*}
\bigl(Q_\theta[v_0,v_1],[v_0,v_1]\bigr)_\HH
&=\frac12\int_\theta^\infty \Bigl(e^{\beta t}E_v(t)+e^{-\beta t}\|{\mathsf P}_m\bigl(\chi Q_t^1\varPhi_v(t)\bigr)\|^2\Bigr)\,dt\\
&\le C\,e^{\beta \theta}E_v(\theta). 
\end{align*}
We can now establish~\eqref{2.7} by literal repetition of the argument used above for problem on the half-line~$\R_+$. 

\smallskip
It remains to establish the uniqueness of solution. Let~$v(t,x)$ be a function that belongs to the space $C(\R_+,H_0^1)\cap C^1(\R_+,L^2)$ and satisfies Eqs.~\eqref{2.1} and~\eqref{2.6} with $\eta(t)=K_b(t)\varPhi_v(t)$ and $v_0=v_1\equiv0$. Since $\eta\in L_{\rm loc}^1(\R_+,L^2)$, inequality~\eqref{1.16} and the boundedness of the operator~$K_b(t)$  imply that
$$
\|\varPhi_v(t)\|_\HH\le C\int_0^t\|\varPhi_v(s)\|_\HH\,ds.
$$
By the Gronwall inequality, we conclude that $v\equiv0$. The proof of Theorem~\ref{t2.1} is complete. 

\section{Observability inequalities}
\label{s3}
This section is devoted to the proof of a truncated observability inequality used in Section~\ref{s2.2}. Namely, let us consider the homogeneous equation
\begin{equation} \label{3.1}
\ddot v-\gamma\dot v-\Delta v+b(t,x)v=0,\quad x\in \Omega,
\end{equation}
supplemented with the Dirichlet boundary condition~\eqref{2.2}. 
We first establish a ``full'' observability inequality for solutions of low regularity and then use a compactness argument to derive the required result. 

\subsection{Observability of low-regularity solutions}
\label{s3.1}

\begin{theorem} \label{t3.1}
Let the hypotheses of Theorem~\ref{t2.1} be fulfilled, let $\chi\in C_0^\infty(\R^3)$ be  such that $\supp\chi\cap\Omega\subset\omega$ and $\chi(x)=1$ for $x\in\Omega_{\delta/2}(x_0)$, and let 
\begin{equation} \label{2.07}
T>2\sup_{x\in\Omega}|x-x_0|,
\end{equation}
where $x_0\in\R^3$ is the point entering Condition~\ref{c2.2}. Then there are positive constants $\sigma_0(r)$ and $M_6=M_6(R,r,T,\gamma,\chi)$ such that, for any initial data $[v_0,v_1]\in\HH^{-\sigma}$ with $0\le\sigma\le \sigma_0(r)$ the solution $v(t,x)$ of problem~\eqref{3.1}, \eqref{2.2}, \eqref{1.12} satisfies the inequality
\begin{equation} \label{3.2}
 \|v_0\|^2_{-\sigma}+\|v_1\|^2_{-\sigma-1}
 \leq M_6\int_0^T \|\chi v(t)\|^2_{-\sigma}\,dt.
\end{equation}
\end{theorem}

\begin{proof}
The proof is divided into three steps: we first establish a unique continuation property for low-regularity solutions; we then use the Bardos--Lebeau--Rauch observability inequality to establish a high-frequency observability; and, finally, these two results are combined to prove the required observability inequality. Without loss of generality, we shall assume that $\gamma=0$; the general case can easily be treated by the change of variable $v(t)=e^{\gamma t/2}w(t)$. 

\medskip
{\it Step~1. Unique continuation property}. 
Let $\sigma\in(0,1)$ be so small that the initial-boundary value problem for~\eqref{3.1} is well posed in~$\HH^{-\sigma}$ and the conclusion of Lemma~\ref{L:com2} is true. We claim that if a solution $v(t,x)$ of problem~\eqref{3.1}, \eqref{2.2}, \eqref{1.12} with initial data $[v_0,v_1]\in\HH^{-\sigma}$ is such that $\chi v=0$ on $J_T\times\Omega$, then $v\equiv0$. Indeed, let $\varphi\in C_0^{\infty}(\R)$ be such that $\varphi=1$ around $0$. For $h\in (0,1]$, let~$v^h$ be the solution of~\eqref{3.1}, \eqref{2.2}  with the initial condition
$$
v^h\bigr|_{t=0}=\varphi(-h^2\Delta)v_0,\quad 
\partial_tv^h\bigr|_{t=0}=\varphi(-h^2\Delta)v_1.
$$
We set $u^h=v^h-\varphi(-h^2\Delta)v$. Then
\begin{equation}
 \label{eq_vh}
\partial_t^2u^h-\Delta u^h+b(t,x)u^h=[\varphi(-h^2\Delta),b]v,
\quad [u^h,\partial_tu^h]\bigr|_{t=0}=0.
\end{equation} 
By Proposition~\ref{p1.3} (with $\sigma=0$) and Lemma~\ref{L:com2}, for all $t\in J_T$ we have
\begin{align}
\|\varPhi_{u^h}(t)\|_{\HH^0}
&\leq M_2\int_0^T \left\|\left[\varphi(-h^2\Delta),b(s)\right]v(s)\right\|_{-1}\,ds
\notag\\
&\leq M_2M_4\int_0^T h^{\sigma}\|v(s)\|_{-\sigma}\,ds
\le C\,h^\sigma\|[v_0,v_1]\|_{\HH^{-\sigma}}.\label{3.5}
\end{align}
Suppose we have shown that
\begin{equation} \label{3.4}
\|\varPhi_{v^h}(t)\|_{\HH^0}\le C\,\bigl(h^\sigma+h^{1-\sigma}\bigr)
\bigl\|[v_0,v_1]\bigr\|_{\HH^{-\sigma}}.
\end{equation}
Then these two inequalities imply that
$$
\lim_{h\to 0}
\bigl(\|\varphi(-h^2\Delta)v_0\|+\|\varphi(-h^2\Delta)v_1\|_{-1}\bigr)=0,
$$
whence we conclude that $v_0=v_1=0$ and, hence, $v\equiv0$. 

\smallskip
To prove~\eqref{3.4}, recall that, by~\cite{DZZ-2008}, the observability inequality~\eqref{3.2} is true with $\sigma=0$. Combining this with~\eqref{1.16}, we derive
\begin{align}
\left\|\varPhi_{v^h}(t)\right\|_{\HH^0}^2
&\le M_2^2 \left\|\varPhi_{v^h}(0)\right\|_{\HH^0}^2
\leq M_2^2M_6\int_0^{T} \|\chi v^h(s)\|^2\,ds\notag\\
&\leq C\int_0^{T}\left\|\chi u^h(s)\right\|^2\,ds
+C\int_0^T \left\|\chi\,\varphi(-h^2\Delta)\,v(s)\right\|^2\,ds. 
\label{est_uniq2}
\end{align} 
By~\eqref{3.5}, the first term on the right-hand side does not exceed $Ch^{2\sigma}\|[v_0,v_1]\|_{\HH^{-\sigma}}^2$. To estimate the second term, we write
$$
\chi\,\varphi(-h^2\Delta)\,v
=\underbrace{\varphi(-h^2\Delta)\chi v}_{=0}+\left[\chi,\varphi(-h^2\Delta)\right]v.
$$
Using Lemma~\ref{L:com1} and Proposition~\ref{p1.3},  we get 
\begin{equation*}
 \label{est_uniq3}
\|\chi \varphi(-h^2\Delta)v(t)\|\leq M_3 h^{1-\sigma}\|v(t)\|_{-\sigma}
\leq M_3M_2h^{1-\sigma}\|[v_0,v_1]\|_{\HH^{-\sigma}}.
\end{equation*} 
Substituting these estimates into~\eqref{est_uniq2}, we obtain the required inequality~\eqref{3.4}.

\smallskip
{\it Step~2. High-frequency observability}. 
We now prove the following weaker version of~\eqref{3.2}:
\begin{equation} \label{3.7}
\sup_{t\in J_T}\bigl\|\varPhi_v(t)\bigr\|_{\HH^{-\sigma}}^2
 \leq C\biggl(\int_0^T \|\chi v(t)\|^2_{-\sigma}\,dt
 +\int_0^T\|v(s)\|_{-\sigma-1}^2ds\biggr).
\end{equation}
To this end, recall that inequality~\eqref{3.2} is true\footnote{The paper~\cite{BLR-1992} deals with the boundary control and establishes the observability inequality~\eqref{3.2} in the scale of Sobolev spaces. On the other hand, the paper~\cite{BLR-1988} is devoted to the case of distributed control and proves~\eqref{3.2} with $\sigma=0$. Even though it is commonly accepted that the observability inequality is true for the scale of Sobolev spaces also for a distributed control, we were not able to find an accurate proof in the literature and outline it in Section~\ref{s3.3} for the reader's convenience.} for solutions of Eq.~\eqref{3.1} with $\gamma=b=0$ (see~\cite{BLR-1992,BLR-1988}). Combining that inequality with~\eqref{1.16}, we derive
$$
\sup_{t\in J_T}\bigl\|S(t)[v_0,v_1]\bigr\|_{\HH^{-\sigma}}^2
 \leq C\int_0^T \|\chi S_0(t)[v_0,v_1]\|^2_{-\sigma}\,dt,
$$
where $S(t)$ and $S_0(t)$ are defined in Section~\ref{s1.2}. 
Using now~\eqref{Duhamel_b} with $\eta\equiv0$, we see that the solution of~\eqref{3.1}, \eqref{2.2}, \eqref{1.12} satisfies the inequality
\begin{align} 
\sup_{t\in J_T}\bigl\|\varPhi_v(t)\bigr\|_{\HH^{-\sigma}}^2
&\le C\int_0^T \biggl\|\chi \Bigl(v(t)-\int_0^tS_0(t-s)[0,(bv)(s)]\,ds\Bigr)\biggr\|^2_{-\sigma}\,dt\notag\\
&\le C\int_0^T \|\chi v(t)\|_{-\sigma}^2dt
+C\int_0^T\|(bv)(t)\|_{-\sigma-1}^2dt.\label{3.8}
\end{align}
If $\sigma>0$ is sufficiently small, then inequalities~\eqref{1.18} and~\eqref{2.5}  and compactness of the embedding $H^{-\sigma}\subset H^{-\sigma-\delta}$ for $\delta>0$ imply that
$$
\|(bv)(t)\|_{-\sigma-1}\le C\,\|v(t)\|_{-\sigma-\delta}
\le\e\,\|v(t)\|_{-\sigma}+C_\e\|v(t)\|_{-\sigma-1},
$$
where $\e>0$ can be chosen arbitrarily small. Substituting this into~\eqref{3.8}, we arrive at the required inequality~\eqref{3.7}.  

\smallskip
{\it Step~3. Conclusion of the proof}. 
We now argue by contradiction. Suppose there is a sequence of solutions $(v^n)$ for~\eqref{3.1}, \eqref{2.2}, \eqref{1.12} such that 
\begin{equation} \label{3.9}
\|\varPhi_{v^n}(0)\|_{\HH^{-\sigma}}=1, \quad 
\int_0^T\|\chi v^n(t)\|_{-\sigma}^2dt\le n^{-1}\quad\mbox{for $n\ge1$}.
\end{equation}
Without loss of generality, we can assume that 
\begin{align*}
v^n&\to v \quad\mbox{weakly$^*$ in $L^\infty(J_T,H^{-\sigma})$},\\
\p_tv^n&\to \p_t v \quad\mbox{weakly$^*$ in $L^\infty(J_T,H^{-\sigma-1})$},
\end{align*}
where $v\in X_T^\sigma$ is a solution of~\eqref{3.1}, \eqref{2.2}. 
Combining this with~\eqref{3.9} and inequality~\eqref{3.7} applied to~$v^n$, we obtain
$$
1\le C\int_0^T\|v(s)\|_{-\sigma-1}^2ds,
$$
so that $v\not\equiv0$. On the other hand, it follows from the second relation in~\eqref{3.9} that 
$$
\int_0^T\|\chi v^n(t)\|_{-\sigma-1}^2dt
=\lim_{n\to\infty}\int_0^T\|\chi v^n(t)\|_{-\sigma-1}^2dt=0,
$$
whence we see that $\chi v\equiv0$. By the unique continuation property established in Step~1, we conclude that $v\equiv0$. The contradiction obtained completes the proof of the theorem.
\end{proof}

\subsection{Truncated observability inequality}
\label{s3.2}
Let us recall that~$H_N$ stands for the vector span of $e_1,\dots,e_N$ and~${\mathsf P}_N$ denotes the orthogonal projection in~$L^2$ to~$H_N$. The following result shows that if~$\varPhi_v(T)$ belongs to~$H_N\times H_N$, then in~\eqref{3.2} the function~$\chi v$ can be replaced by its projection to~$H_m$ with a sufficiently large~$m$.  

\begin{theorem} \label{t3.2}
Under the hypotheses of Theorem~\ref{t3.1}, for any $N\ge1$ there is an integer $m\ge1$ such that if $v\in X_T^\sigma$ is a solution of~\eqref{3.1}, \eqref{2.2} satisfying the condition $\varPhi_v(T)\in H_N\times H_N$, then 
\begin{equation} \label{3.10}
 \|\varPhi_v(0)\|_{\HH^{-\sigma}}^2
 \leq 2M_6\int_0^T \|{\mathsf P}_m(\chi v(t))\|^2_{-\sigma}\,dt.
\end{equation}
\end{theorem}

\begin{proof}
We repeat the argument used in~\cite{BRS-2011} for the case of the linearised Navier--Stokes system. It suffices to prove that if $v\in X_T^\sigma$ is a solution of~\eqref{3.1}, \eqref{2.2} satisfying the condition $\varPhi_v(T)\in H_N\times H_N$, then 
\begin{equation} \label{3.11}
\int_0^T \|\chi v(t)\|^2\,dt
\le C \int_0^T \|\chi v(t)\|^2_{-\sigma}\,dt,
\end{equation}
where $C>0$ depends only on~$N$ and~$R$; see Section~A.3 in~\cite{BRS-2011}. We argue by contradiction. 

Suppose there are sequence $(v^n)\subset X_T^\sigma$ and $(b^n)\subset L^\infty(J_T,H^r\cap L^\infty)$ such that
\begin{gather}
\varPhi_{v^n}(T)\in H_N\times H_N, \quad \|\varPhi_{v^n}(T)\|_{\HH^{-\sigma}}=1, \quad \barr b^n\barr\le R,
\label{3.12}\\
\ddot v_n-\gamma\dot v^n-\Delta v^n+b^n(t,x)v^n=0,
\label{3.13}\\
\int_0^T \|\chi v^n(t)\|^2\,dt
\ge n \int_0^T \|\chi v^n(t)\|^2_{-\sigma}\,dt.
\label{3.14}
\end{gather}
Passing to a subsequence, we can assume that there are function~$v$ and~$b$ such that $v\in L^\infty(J_T,H^{-\sigma})$, $\ddot v\in L^\infty(J_T,H^{-\sigma-2})$, $b\in L^\infty(J_T,H^r\cap L^\infty)$, and the following convergences hold: 
\begin{align*}
\varPhi_{v^n}(T)&\to \varPhi_{v}(T),\\
v^n&\to v\quad\mbox{weakly$^*$ in $L^\infty(J_T,H^{-\sigma})$},\\
\dot v^n&\to \dot v\quad\mbox{weakly$^*$ in $L^\infty(J_T,H^{-\sigma-1})$},\\
\ddot v^n&\to \ddot v\quad\mbox{weakly$^*$ in $L^\infty(J_T,H^{-\sigma-2})$},\\
b^n&\to b\quad\mbox{weakly$^*$ in $L^\infty(J_T\times\Omega)$ and $L^\infty(J_T,H^{r})$}.
\end{align*}
It follows from the first two relations in~\eqref{3.12} that $0\ne\varPhi_v(T)\in\HH$ and $\|\varPhi_{v^n}(T)\|_\HH\le C$. By the uniqueness of solution for~\eqref{2.1}, the functions~$v^n$ must belong to~$X_T^0$, and their norms are uniformly bounded. Inequality~\eqref{3.14} now implies that $\|\chi v^n\|_{L^2(J_T,H^{-\sigma})}\to0$, whence we conclude that $\chi v\equiv0$. Suppose we have proved that that we can pass to the limit as $n\to\infty$ in Eq.~\eqref{3.13}. Then the function~$v$ is a solution of the limiting equation~\eqref{3.1} and, hence, the observability inequality~\eqref{3.2} holds for it. It follows that $v\equiv0$ and therefore $\varPhi_v(T)=0$. The contradiction obtained proves the required inequality~\eqref{3.11}.

\smallskip
It remains to prove that one can pass to the limit in~\eqref{3.13}. The only nontrivial term is $b^nv^n$. We need to show that 
\begin{equation} \label{limit}
\langle b^nv^n,\varphi\rangle \to \langle bv,\varphi\rangle 
\quad\mbox{for any $\varphi\in C_0^\infty$},
\end{equation}
when $C_0^\infty$ is considered on the open set $(0,T)\times\Omega$, and  $\langle\cdot,\cdot\rangle$ denotes the duality between the space of distributions and~$C_0^\infty$. It follows from Lemma~5.1 in~\cite{lions1969} the sequence~$\{v^n\}$ is relatively compact in the space $L^p(J_T,H^{-\sigma-\e})$ for any $p\in[1,\infty)$ and $\e>0$. Therefore so is the sequence~$\{v^n\varphi\}$. Choosing~$\sigma$ and~$\e$ so small that $\sigma+\e\le r$, we now write
$$
\langle b^nv^n,\varphi\rangle
=\int_0^T\bigl(b^n(t,\cdot),(v^n\varphi)(t,\cdot)\bigr)\,dt.
$$
Since the weak$^*$ convergence in $L^\infty(J_T,H^r)$ is uniform on compact subsets of~$L^1(J_T,H^{-r})$, the above-mentioned compactness of~$\{v^n\varphi\}$ implies that
$$
\int_0^T(b^n,v^n\varphi)\,dt=
\int_0^T(b^n-b,v^n\varphi)\,dt+\int_0^T(b\varphi,v^n)\,dt
\to \int_0^T(b\varphi,v)\,dt,
$$
where we used the fact that $b\varphi\in L^\infty(J_T,H^r)\subset L^1(J_T,H^\sigma)$. 
This completes the proof of~\eqref{limit} and that of Theorem~\ref{t3.2}. 
\end{proof}

\subsection{Observability inequality for the wave equation}
\label{s3.3}
In this section, we prove that if $[v_0,v_1]\in\HH^{-\sigma}$ with $\sigma\in(0,\frac12)$, then the solution~$v(t,x)$ of problem~\eqref{3.1}, \eqref{2.2}, \eqref{1.12} with $\gamma=0$ and $b\equiv0$ satisfies~\eqref{3.2}. To this end, we first recall that $\HH^s=\HH_D^s$ for $\frac12<s<\frac32$, and therefore, by Theorem~1.3 in~\cite{DL-2009}, for any $[u_0,u_1]\in\HH^{\sigma+1}$ there is $\zeta\in L^2(J_T,H^\sigma)$ such that 
\begin{equation} \label{3.15}
\bigl\|\zeta\bigr\|_{L^2(J_T,H^{\sigma})}\le C\bigl(\|u_0\|_{\sigma+1}+\|u_1\|_{\sigma}\bigr), 
\end{equation}
and the solution $u\in C(J_T,H_0^{\sigma+1})\cap C^1(J_T,H^\sigma)$ of the problem 
$$
\ddot u-\Delta u=\chi(x)\zeta(t,x), \quad u(0,x)=u_0(x), \quad \dot u(0,x)=u_1(x), 
$$
satisfies the relation
\begin{equation} \label{3.16}
u(T,x)= \dot u(T,x)\equiv0. 
\end{equation}
Now let $v\in C(J_T,H^{-\sigma})\cap C^1(J_T,H^{-\sigma-1})$ be the solution of~\eqref{3.1}, \eqref{2.2}, \eqref{1.12} with $[v_0,v_1]\in \HH^{-\sigma}$. Then, in view of~\eqref{3.16}, we have 
\begin{align*}
(u_0,v_1)-(u_1,v_0)&=\int_0^T\frac{d}{dt}\bigl((\dot u,v)-(u,\dot v)\bigr)\,dt
=\int_0^T\bigl((\ddot u,v)-(u,\ddot v)\bigr)\,dt\\
&=\int_0^T\bigl((\Delta u+\chi\zeta,v)-(u,\Delta v)\bigr)\,dt=\int_0^T(\zeta,\chi v)\,dt\\
&\le \bigl\|\zeta\bigr\|_{L^2(J_T,H^{\sigma})}\,\bigl\|\chi v\bigr\|_{L^2(J_T,H^{-\sigma})}. 
\end{align*}
Taking $u_0=(-\Delta)^{-\sigma-1}v_1\in H_0^{\sigma+1}$ and $u_1=-(-\Delta)^{-\sigma}v_0\in H^{\sigma}$ and using~\eqref{3.15}, we arrive at the required inequality~\eqref{3.2}.

\section{Main result: stabilisation of the non-linear problem}
\label{s4}
Let us consider the nonlinear problem~\eqref{1}--\eqref{3}, where~$\Omega\subset\R^3$ is a bounded domain with a $C^2$-smooth boundary~$\Gamma$. We assume that $\gamma>0$, and the function $f\in C^2(\R)$ satisfies conditions~\eqref{1.2} and~\eqref{1.5}. Let us consider a solution~$\hat u(t,x)$ with initial data $[\hat u_0,\hat u_1]\in (H_0^1\cap H^2)\times H_0^1$ and a right-hand side~$h\in W^{1,\infty}(\R_+,L^2)$. 
By Proposition~\ref{p1.2}, there is $C>0$ such that 
\begin{equation} \label{4.1}
\|\hat u(t,\cdot)\|_2+\|\p_t\hat u(t,\cdot)\|_1\le C\quad\mbox{for all $t\ge0$}. 
\end{equation}
Let us define a function $b$ by the relation $b(t,x)=f'(\hat u(t,x))$. It follows from~\eqref{4.1} and the conditions imposed on~$f$ that~$b$ satisfies~\eqref{2.5} with $r=1$.  Therefore, if Conditions~\ref{c2.1} and~\ref{c2.2} are also satisfied, then Theorem~\ref{t2.1} is applicable, and one can construct a feedback law~$K_{\hat u}(t):=K_b(t)$ exponentially stabilising the linearised problem~\eqref{2.1}, \eqref{2.2}. The following theorem, which is the main result of this paper, shows that the same law stabilises locally exponentially also the nonlinear problem. 

\begin{theorem} \label{t4.1}
Under the above hypotheses, there are positive constants~$C$ and~$\e$ such that, for any initial data $[u_0,u_1]\in H_0^1\times L^2$ satisfying the inequality
\begin{equation} \label{4.2}
\|u_0-\hat u_0\|_1+\|u_1-\hat u_1\|\le\e,
\end{equation}
problem~\eqref{4}, \eqref{2}, \eqref{3} with 
$\eta(t,x)=K_{\hat u}(t)\varPhi_{u-\hat u}(t)$ has a unique solution
$$
u\in \XX:=C(\R_+,H_0^1)\cap C^1(\R_+,L^2),
$$
for which  inequality~\eqref{5} holds.
\end{theorem}

\begin{proof}
The proof based on a fixed point argument is rather standard  (cf.\ Section~4 in~\cite{BRS-2011}), and therefore we shall only outline it. 

\smallskip
We seek a solution of the form $u=\hat u+v$. Then $v$ must be a solution of the problem
\begin{align} 
\ddot v+\gamma\dot v-\Delta v+f(\hat u+v)-f(\hat u)&=\eta(t,x),\label{4.3}\\
v\bigr|_{\p \Omega}&=0,\label{4.4}\\
v(0,x)=v_0(x), \quad \p_tv(0,x)&=v_1(x), 
\label{4.5}
\end{align}
where $v_0=u_0-\hat u_0$, $v_1=u_1-\hat u_1$, and $\eta(t)=K_{\hat u}(t)\varPhi_v(t)$. Let us fix $\theta>0$ and define the metric space
$$
\ZZ_\theta:=\{v\in \XX:\varPhi_v(0)=[v_0,v_1],E_{v}(t)\le \theta\,e^{-\beta t}E_{v}(0)\mbox{ for $t\ge0$}\}. 
$$
If we construct a solution $v\in\ZZ_\theta$, then the corresponding function $u=\hat u+v$ will be the required solution of the original problem. The fact that there are no other solutions can easily proved by a standard argument (e.g., see Chapter~1 of~\cite{lions1969}). 

\smallskip
{\it Step~1}. 
Let us endow~$\ZZ_\theta$ with the metric generated by the norm
$$
\|v\|_\ZZ=\sup_{t\ge0}\bigl(e^{\beta t}E_v(t)\bigr)^{1/2}.
$$
Define a mapping $\Xi:\ZZ_\theta\to\XX$ that takes $w\in\ZZ_\theta$ to the solution of~\eqref{2.1}--\eqref{1.12}, in which
$$
b(t,x)=f'(\hat u(t,x)), \quad
 \eta(t,x)=K_{\hat u}(t)\varPhi_v(t)-\bigl(f(\hat u+w)-f(\hat u)-b(t,x)w\bigr).
$$
The mapping~$\Xi$ is well defined. Indeed, the homogeneous problem (that is system~\eqref{2.1}, \eqref{2.6} with $\eta=K_{\hat u}(t)\varPhi_v(t)$) is well posed in view of Theorem~\ref{t2.1}, while a solution of the inhomogeneous equation can be written in the form of the Duhamel integral. Suppose we have shown that, for an appropriate choice of~$\theta$, the mapping~$\Xi$ is a contraction in~$\ZZ_\theta$. Then the unique fixed point $v\in\ZZ_\theta$ for~$\Xi$ is a solution of~\eqref{4.3}--\eqref{4.5}, and it satisfies~\eqref{5} with $C=\theta$. 

\smallskip
{\it Step~2}. 
Let us prove that~$\Xi$ maps the space~$\ZZ_\theta$ into itself. Define $\HH=H_0^1\times L^2$ and  denote by $U(t,s):\HH\to\HH$ the operator that takes $[v_0,v_1]$ to~$v(t)$, where $v(t,x)$ is the solutions of~\eqref{2.1}, \eqref{2.2}, \eqref{2.6} with $\eta(t)=K_{\hat u}(t)\varPhi_v(t)$. Then we can write
\begin{equation} \label{4.6}
(\Xi w)(t)=U(t,0)[v_0,v_1]-\int_0^tU(t,s)[0,g(s)]\,ds,
\end{equation}
where we set $g(t,x)=f(\hat u+w)-f(\hat u)-f'(\hat u)w$. 
By Theorem~\ref{t2.1}, the operator norm of~$U(t,s)$ satisfies the inequality
\begin{equation} \label{4.7}
\|U(t,s)\|_\LL^2\le C\,e^{-\beta(t-s)}, \quad t\ge s\ge0. 
\end{equation}
Combining this with~\eqref{4.6}, we see that
\begin{align}
\|(\Xi w)(t)\|_\HH^2
&\le 2\,\|U(t,0)\|_\LL^2\bigl\|[v_0,v_1]\bigr\|_\HH^2
+2\,\biggl(\int_0^t\|U(t,s)\|_\LL\|g(s)\|\,ds\biggr)^2\notag\\
&\le2Ce^{-\beta t} \biggl(\bigl\|[v_0,v_1]\bigr\|_\HH^2
+\Bigl(\,\int_0^te^{\beta s/2}\|g(s)\|\,ds\Bigr)^2\biggr).
\label{4.8}
\end{align}
Now note that (see Section~4 in~\cite{BRS-2011})
$$
\sup_{t\ge0}\Bigl(\,\int_0^te^{\beta s/2}\|g(s)\|\,ds\Bigr)^2
\le C_1\sup_{t\ge0}\int_t^{t+1}e^{2\beta s}\|g(s)\|^2ds. 
$$
Substituting this into~\eqref{4.8}, we derive
\begin{equation} \label{4.9}
\|\Xi w\|_{\ZZ}^2\le 2\bigl\|[v_0,v_1]\bigr\|_\HH^2
+C_1\sup_{t\ge0}\int_t^{t+1}e^{2\beta s}\|g(s)\|^2ds.
\end{equation}
It follows from the conditions imposed on~$f$, the Taylor expansion, and inequality~\eqref{4.1} that, for any $w\in\ZZ_\theta$ and $[v_0,v_1]\in\HH$ satisfying~\eqref{4.2}, we have
\begin{align*}
\|g(s)\|^2
&\le C_2\bigl(1+\|\hat u(s)\|_1^2+\|w(s)\|_1^2\bigr)\,\|w(s)\|_1^4\\
&\le C_3\theta\e^2(1+\theta\e^2)e^{-2\beta s}\bigl\|[v_0,v_1]\bigr\|_\HH^2. 
\end{align*}
Combining this with~\eqref{4.9} and assuming that~$\theta$ is sufficiently large and $\theta\e^2\le1$, we see that $\Xi(\ZZ_\theta)\subset\ZZ_\theta$.

{\it Step 3}. It remains to prove that~$\Xi$ is a contraction. It follows from~\eqref{4.6} that if $w_1,w_2\in\ZZ_\theta$, then
$$
(\Xi w_1-\Xi w_2)(t)=\int_0^tU(t,s)[0,h(s)]\,ds,
$$
where $h(s)=f(\hat u+w_2)-f(\hat u+w_1)-f'(\hat u)(w_2-w_1)$. Repeating the argument used in the derivation of~\eqref{4.9}, we obtain
\begin{equation} \label{4.10}
\|\Xi w_1-\Xi w_2\|_{\ZZ}^2\le C_1\sup_{t\ge0}\int_t^{t+1}e^{2\beta s}\|h(s)\|^2ds.
\end{equation}
Using the mean value theorem, we easily show that 
$$
\|h(s)\|^2\le C_4\|w\|_1^2 \bigl(\|w_1\|_1^2+\|w_2\|_1^2\bigr)
 \bigl(1+\|w_1\|_1^2+\|w_2\|_1^2\bigr),
$$
where $w=w_1-w_2$. Substituting this into~\eqref{4.10} and using~\eqref{4.2}, we derive
$$
\|\Xi w_1-\Xi w_2\|_{\ZZ}^2\le C_5\theta\e^2\|w\|_\ZZ^2.
$$
Hence, if $\e>0$ is sufficiently small, then $\Xi$ is a contraction. This completes the proof of Theorem~\ref{t4.1}. 
\end{proof}

\addcontentsline{toc}{section}{Bibliography}
\def\cprime{$'$} \def\cprime{$'$}
  \def\polhk#1{\setbox0=\hbox{#1}{\ooalign{\hidewidth
  \lower1.5ex\hbox{`}\hidewidth\crcr\unhbox0}}}
  \def\polhk#1{\setbox0=\hbox{#1}{\ooalign{\hidewidth
  \lower1.5ex\hbox{`}\hidewidth\crcr\unhbox0}}}
  \def\polhk#1{\setbox0=\hbox{#1}{\ooalign{\hidewidth
  \lower1.5ex\hbox{`}\hidewidth\crcr\unhbox0}}} \def\cprime{$'$}
  \def\polhk#1{\setbox0=\hbox{#1}{\ooalign{\hidewidth
  \lower1.5ex\hbox{`}\hidewidth\crcr\unhbox0}}} \def\cprime{$'$}
  \def\cprime{$'$} \def\cprime{$'$} \def\cprime{$'$}
\providecommand{\bysame}{\leavevmode\hbox to3em{\hrulefill}\thinspace}
\providecommand{\MR}{\relax\ifhmode\unskip\space\fi MR }
\providecommand{\MRhref}[2]{%
  \href{http://www.ams.org/mathscinet-getitem?mr=#1}{#2}
}
\providecommand{\href}[2]{#2}

\end{document}